\documentclass[12pt,russian]{amsart}
\usepackage[all,matrix,arrow]{xy}
\usepackage{mathrsfs, amssymb, array, amsmath, upgreek, amsfonts}
\usepackage{amsthm}
\usepackage[mathcal]{euscript}
\usepackage{float}
\usepackage{hyperref}
\usepackage{verbatim}
\usepackage{tabularx,multirow,makecell,longtable}

\usepackage[T1,T2A]{fontenc}
\usepackage[utf8]{inputenc}
\usepackage[english]{babel}

\makeatletter
\@addtoreset{equation}{section}
\makeatother

\newcounter{NNN}[table]\numberwithin{NNN}{table}

\newcommand{\CC}{\mathbb C}
\newcommand{\PP}{\mathbb P}

\newcommand{\ZZ}{\mathbb Z}

\newcommand{\MMM}{\mathscr{M}}

\newcommand{\EEE}{\mathscr{E}}
\newcommand{\OOO}{\mathscr{O}}
\newcommand{\LLL}{\mathscr{L}}

\newcommand{\FFF}{\mathscr{F}}

\newcommand{\KKK}{{\mathscr{K}}}

\newcommand{\pg}{\mathrm{p_g}}
\newcommand{\pa}{\mathrm{p_a}}
\newcommand{\h}{\mathrm{h}}

\newcommand{\Bim}{\operatorname{Bim}}

\newcommand{\Graph}{\operatorname{Graph}}

\newcommand{\ind}{\operatorname{ind}}

\newcommand{\hol}{\operatorname{hol}}

\newcommand{\Aut}{\operatorname{Aut}}
\newcommand{\ab}{\mathrm{a}}

\newcommand{\comp}{\circ}

\newcommand{\rk}{\operatorname{rk}}

\newcommand{\GL}{\operatorname{GL}}

\def \ge {\geqslant}
\def \le {\leqslant}

\theoremstyle{plain}
\newtheorem{theorem}[subsection]{Theorem}
\newtheorem{lemma}[subsection]{Lemma}
\newtheorem{proposition}[subsection]{Proposition}
\newtheorem{corollary}[subsection]{Corollary}

\theoremstyle{definition}
\newtheorem{definition}[subsection]{Definition}

\newtheorem{remark}[subsection]{Remark}
\newtheorem{remark-definition}[subsection]{Remark-Definition}

\title{Finite groups of bimeromorphic selfmaps of uniruled K\"ahler threefolds}

\author{Yu.~G.~Prokhorov, \quad C.~A.~Shramov}

\address{\emph{Yuri Prokhorov}
\newline
\textnormal{Steklov Mathematical Institute of RAS,
8 Gubkina street, Moscow 119991, Russia.
}
\newline
\textnormal{
HSE University, Russian Federation,
Laboratory of Algebraic Geometry, 6 Usacheva str., Moscow, 119048, Russia.
}
\newline
\textnormal{\texttt{prokhoro@mi-ras.ru}}}

\address{\emph{Constantin Shramov}
\newline
\textnormal{Steklov Mathematical Institute of RAS,
8 Gubkina street, Moscow 119991, Russia.
}
\newline
\textnormal{
HSE University, Russian Federation,
Laboratory of Algebraic Geometry, 6 Usacheva str., Moscow, 119048, Russia.
}
\newline
\textnormal{\texttt{costya.shramov@gmail.com}}}

\thanks{This work is supported by the Russian Science Foundation under grant \textnumero 18-11-00121.}
\date{}

\begin{document}
\maketitle

\begin{abstract}
We classify uniruled compact K\"ahler threefolds whose groups of bimeromorphic selfmaps do not have Jordan property.
\end{abstract}

\tableofcontents

\section{Introduction}

Groups of automorphisms and bimeromorphic selfmaps of complex manifolds can have a very complicated
structure. In many cases it is relatively easy to study them on the level of finite subgroups.
Although even in the most simple situations such groups can contain infinitely many
non-isomorphic finite subgroups, it often happens that certain important parameters of these
subgroups are bounded. An example of such a behavior is provided by \emph{Jordan property}.

\begin{definition}[{see \cite[Definition~2.1]{Popov2011}}]
\label{definition:Jordan}
A group~$\Gamma$ is called \emph{Jordan}
(alternatively, one says that $\Gamma$ has \emph{Jordan property}),
if there exists a constant $J=J(\Gamma)$ such that any finite subgroup~\mbox{$G\subset\Gamma$}
contains a normal abelian subgroup $A\subset G$ of index at most~$J$.
\end{definition}

An old theorem due to C.\,Jordan states that the groups
$\GL_n(\CC)$ enjoy this property (see for instance~\mbox{\cite[Theorem~36.13]{Curtis-Reiner-1962}}).
J.-P.\,Serre pointed out that this is also the case for certain groups of geometric origin; namely,
he proved (\cite[Theorem~5.3]{Serre2009},
\mbox{\cite[Th\'eor\`eme~3.1]{Serre-2008-2009}}),
that the group of birational selfmaps of the projective plane over a field of zero characteristic is Jordan.
Yu.\,G.\,Zarhin~\cite{Zarhin10} found an example of an algebraic surface whose group of birational selfmaps
is not Jordan, and V.\,L.\,Popov classified all such surfaces.

\begin{theorem}[{\cite[Theorem~2.32]{Popov2011}}]
\label{theorem:Popov}
Let $X$ be an algebraic surface over the field of complex numbers~$\CC$.
Then the group of birational selfmaps of $X$ is not Jordan if and only if~$X$ is birational to a direct product~\mbox{$E\times\PP^1$}, where $E$
is an elliptic curve.
\end{theorem}

There are certain results concerning Jordan property for birational automorphism groups of
higher dimensional algebraic varieties (see \cite[Theorem~1.8]{ProkhorovShramov-RC},
\cite[Theorem~1.8]{Prokhorov-Shramov-2013}, \mbox{\cite[Theorem~1.1]{Birkar}}).
Furthermore, for birational automorphism groups of the projective plane and the three-dimensional
projective space the bounds for the corresponding constants are known (see \cite{Yasinsky2016a}
and \cite{Prokhorov-Shramov-JCr3}).
Some of these bounds can be made more precise if instead of arbitrary finite groups one considers 
a more restricted class of groups, for instance, finite $p$-groups, see~\cite{ProkhorovShramov-p-groups}
and~\cite[\S3.3]{Popov-3plots}. Also, there are numerous results on Jordan property for diffeomorphism groups of smooth manifolds, and other groups
of this kind, see \cite{Popov-Diff}, \cite{CsikosPyberSzabo}, \cite{Riera2016}, \cite{Riera-Spheres}, \cite{Riera-OddCohomology},
\cite{Riera-Symp}, \cite{Riera-HamSymp}, and~\cite{Ye2017}.

For algebraic threefolds one has the following analog of Theorem~\ref{theorem:Popov}.

\begin{theorem}[{\cite[Theorem~1.8]{Prokhorov-Shramov-3folds}}]
\label{theorem:dim-3}
Let $X$ be a three-dimensional algebraic variety over $\CC$.
Than its group of birational selfmaps is not Jordan if and only if
$X$ is birational either to $E\times\PP^2$, where $E$ is an elliptic curve, or to $S\times\PP^1$, where $S$ is a surface of one of the
following types:
\begin{itemize}
\item
an abelian surface;

\item
a bielliptic surface;

\item
a surface of Kodaira dimension $1$ such that the Jacobian fibration of its pluricanonical
fibration is locally trivial \textup(in Zariski topology\textup).
\end{itemize}
\end{theorem}

Recently, there were attempts to study the groups of automorphisms and bimeromorphic selfmaps
of complex manifolds from the point of view of Jordan property. In particular,
in~\cite{ProkhorovShramov-Surfaces}
(see also~\cite{ProkhorovShramov-IK})
the authors obtained a generalization of Theorem~\ref{theorem:Popov}
for the case of compact complex surfaces.
Jordan property for automorphism groups of three-dimensional Moishezon compact complex manifolds
was proved in~\cite{Pr-Sr:Moishezon}.
However, for arbitrary compact complex manifolds of higher dimension
the situation is still unclear because of the lack of appropriate techniques
that would allow to work with their automorphism groups. On the other hand, it
is known that compact K\"ahler manifolds exhibit many similarities with
algebraic varieties, in particular on the level of automorphism groups
(see~\cite{Kim}).

In this paper we prove the following result adjacent to Theorem~\ref{theorem:dim-3}.

\begin{theorem}
\label{theorem:main}
Let $X$ be a non-algebraic three-dimensional uniruled compact K\"ahler manifold.
Suppose that its group of bimeromorphic selfmaps is not Jordan.
Then $X$ is bimeromorphic to the projectivization of a holomorphic vector bundle
of rank $2$ on a two-dimensional complex torus $S$ of algebraic dimension $1$.
If moreover the algebraic dimension of $X$ equals $2$, then $X$ is bimeromorphic to the direct product~\mbox{$\PP^1\times S$}.
\end{theorem}

\begin{remark}
Let $S$ be a complex torus of positive algebraic dimension. Then the group of
bimeromorphic selfmaps of $\PP^1\times S$ is not Jordan, see
\cite[Theorem~1.9]{Zarhin-Tori}. Furthermore, there are examples of (non-trivial) decomposable
holomorphic vector bundles $\EEE$ of rank~$2$ on~$S$ such that the group of bimeromorphic selfmaps
of the projectivization of $\EEE$ is not Jordan, see \cite[Theorem~1.10]{Zarhin-Tori}
and~\mbox{\cite[Theorem~1.12]{Zarhin-Tori}}. We do not know whether one can choose such a projectivization
so that it has algebraic dimension $1$ in the case when~\mbox{$\dim S=2$}, and whether one can construct
an example like this with an indecomposable vector bundle~$\EEE$.
\end{remark}

\smallskip
In Section~\ref{section:prelim} we recall the necessary auxiliary assertions.
In Section~\ref{section:uniruled} we recall the definitions and the basic properties of uniruled
and rationally connected manifolds, and also the basic properties of the maximal rationally connected fibration.
In Section~\ref{section:MRC}
we study the interaction between the maximal rationally connected fibration and the
group of bimeromorphic selfmaps of a compact complex manifold. In Section~\ref{section:conic-bundles}
we study the properties of conic bundles over non-algebraic compact complex surfaces.
Many results of this section are stated and proved in a more general form than we
actually need in the present paper; we hope that this will fill the existing gap in the literature.
In Section~\ref{section:a-0} we study the groups of bimeromorphic selfmaps of three-dimensional compact complex manifolds
fibred into rational curves over a surface of algebraic dimension~$0$.
In Section~\ref{section:a-1-1} we study the groups of bimeromorphic selfmaps of three-dimensional compact complex manifolds
fibred into rational curves over a surface of algebraic dimension~$1$, and complete the proof of Theorem~\ref{theorem:main}.

We use the following notation and conventions.
A \emph{complex manifold} is an irreducible smooth reduced complex space.
A \textit{morphism} is a holomorphic map of complex manifolds. By $\ab(X)$ we denote the algebraic dimension
of a compact complex manifold $X$, and by $\varkappa(X)$ we denote its Kodaira dimension.
For a complex manifold~$Z$ by its \emph{typical point}
we mean a point from a non-empty subset of the form~\mbox{$Z\setminus\Delta$},
where~$\Delta$ is a closed analytic subset in $Z$.
A typical fiber of a morphism of complex manifolds is defined in a similar way.

Let $\tau\colon X\dasharrow Y$ be a meromorphic map of complex manifolds, and let $\gamma\colon X\dasharrow X$ be a bimeromorphic map;
we will say that the action of $\gamma$ is \emph{fiberwise with respect to $\tau$}, if for every fiber $F$ of the map $\tau$
such that $\gamma$ is defined at least at one point of $F$, the image of every such point under $\gamma$
is again contained in~$F$.

If $X$ is a compact complex manifold, then
$\KKK_X$ will denote the canonical line bundle on $X$.
By~$\Omega_X^p$ we denote the vector bundle of holomorphic $p$-forms on~$X$.
The \emph{Hodge numbers}
$\h^{p,q}(X)$ are defined as the dimensions of the cohomology groups~\mbox{$H^q(X, \Omega^p_{X})$}.
In particular, if~\mbox{$\dim X=n$}, then the \emph{geometric genus} of $X$ is defined as~\mbox{$\pg(X)=\h^{n,0}(X)$}.

By $\pa(D)$ we denote the arithmetic genus of a projective scheme $D$.

\smallskip
We are grateful to T.\,Bandman, F.\,Campana, S.\,Nemirovski, and
Yu.\,Zarhin for useful discussions, as well as to the referee for
several important comments that helped us improve the manuscript.

\section{Preliminaries}
\label{section:prelim}

In this section we recall the necessary facts about complex manifolds and meromorphic maps.

Let us say that a group $\Gamma$ has \emph{bounded finite subgroups} if there exists a constant $B=B(\Gamma)$
such that the order of any finite subgroup of $\Gamma$ does not exceed~$B$.

\begin{lemma}
\label{lemma:group-theory}
Let
$$
1\longrightarrow\Gamma'\longrightarrow\Gamma\longrightarrow\Gamma''
$$
be an exact sequence of groups.
Suppose that the group $\Gamma''$ has bounded finite subgroups.
Then the group $\Gamma$ is Jordan (respectively, has bounded finite subgroups)
if and only if the group $\Gamma'$ is Jordan (respectively, has bounded finite subgroups).
\end{lemma}
\begin{proof}
Obvious.
\end{proof}

We will need some basic facts about bimeromorphic geometry of complex manifolds. The reader can find more details
on this in \cite{Ueno1975}
and \cite{several-complex-variables-7}.

\begin{definition}
A proper morphism $f\colon X\to Y$ of (not necessary compact) complex manifolds is called a (proper) \emph{modification},
if there exist closed analytic subsets
$V\subsetneq X$ and~\mbox{$W \subsetneq Y$} such that $f$ induces an isomorphism
$X\setminus V\cong Y\setminus W$.
\end{definition}

\begin{definition}
Let $X$ and $Y$ be complex manifolds.
A \emph{meromorphic map} $f\colon X \dashrightarrow Y$ is a map
$X\to 2^Y$ from the set $X$ to the set $2^Y$ of subsets of $Y$ such that its graph
\[
\Graph_f=\{ (x,y) \mid y\in f(x)\} \subset X\times Y
\]
is an irreducible closed analytic subset of the complex manifold $X\times Y$,
and the first projection
$$
\operatorname{pr}_X\colon \Graph_f\to X
$$
is a modification.
For every meromorphic map~\mbox{$f\colon X \dashrightarrow Y$}
there exists the minimal closed analytic subset $V\subset X$
such that the restriction~$f|_{X\setminus V}$ is holomorphic.
This subset is called the \emph{indeterminacy locus} of the map~$f$.
A \emph{typical fiber} of $f$ is a typical fiber of the projection
$$
\operatorname{pr}_Y\colon \Graph_f\to Y.
$$
A meromorphic map is called \emph{bimeromorphic} if the projection $\operatorname{pr}_Y$
is also a modification. The set of all bimeromorphic maps $X \dashrightarrow X$ is a group
which we will denote by~\mbox{$\Bim(X)$}.
\end{definition}

Note that in the case when $X$ and $Y$ are smooth complex projective algebraic varieties,
according to the GAGA principle the graph~\mbox{$\Graph_f$}
of a meromorphic map $f\colon X \dashrightarrow Y$
is an algebraic subvariety of~\mbox{$X\times Y$}, and thus $f$ is also a rational map in this case. In particular,
for a smooth complex projective algebraic variety~$X$ the
group~\mbox{$\Bim(X)$}
defined above coincides with the group of birational automorphisms of~$X$.

It is clear that a bimeromorphic map $f\colon X \dashrightarrow Y$ of compact complex manifolds induces an isomorphism of the fields of meromorphic functions
$$
f^*\colon \MMM(Y)\stackrel{\sim}\longrightarrow \MMM(X).
$$
However, unlike the algebraic case, such an isomorphism does not usually
define the map $f$.

Recall that a complex manifold is said to be \emph{K\"ahler},
if it has a Hermitian metric such that the corresponding $(1,1)$-form $\omega$ is closed; in this case $\omega$ is called the \emph{K\"ahler form}.
All complex tori are K\"ahler manifolds. Examples of K\"ahler manifolds covered by rational curves can be obtained from
the following statement.

\begin{theorem}[{\cite[Proposition~3.18]{Voisin2007a}}]
\label{theorem:Voisin}
Let $X$ be a compact K\"ahler manifold, and let $\EEE$ be a holomorphic vector bundle on~$X$.
Then the projectivization of $\EEE$ is a K\"ahler manifold.
\end{theorem}

We need the following sufficient condition
for algebraicity of K\"ahler manifolds.

\begin{proposition}\label{prop:Kahler-proj}
Let $X$ be a compact K\"ahler manifold.
Assume that~\mbox{$H^0(X,\Omega_X^2)=0$}. Then $X$ is a projective
algebraic  variety.
\end{proposition}

\begin{proof}
In this case $H^2(X,\CC)=H^{1,1}(X)$ and so the $(1,1)$-form associated
with a K\"ahler metric is integral (i.\,e. $X$ is a Hodge manifold).
Therefore, it is projective by the Kodaira criterion \cite{Kodaira1954}.
\end{proof}

We will also use one general sufficient condition for algebraicity.
Recall that an $n$-dimensional complex manifold is called
\emph{Moishezon} if the transcendence degree of its field of meromorphic
functions equals~$n$. It is well-known that a compact K\"ahler Moishezon
manifold is a projective algebraic variety,
see~\mbox{\cite[Theorem~11]{Moishezon-1967eng}}.

\begin{lemma}[{see \cite[Proposition~12.2]{Ueno1975}}]
\label{lemma:projective-easy}
Let $X$ be a compact complex manifold, $Y$ be a Moishezon manifold,
and $h\colon X\to Y$ be a morphism. Assume that a typical fiber of $h$ is
a  rational curve.
Then the manifold~$X$ is Moishezon. If moreover $X$ is K\"ahler, then
it is a projective algebraic variety.
\end{lemma}

Automorphism groups of K\"ahler manifolds have nice properties.

\begin{theorem}[{\cite{Kim}}]
\label{theorem:aut}
Let $X$ be a compact K\"ahler manifold. Then the group $\Aut(X)$ is Jordan.
\end{theorem}

Jordan property is also known to hold for automorphism groups of certain
special compact complex manifolds.

\begin{lemma}[{see \cite[Corollary~5.9]{Shramov-Fiberwise}}]
\label{lemma:P1-bundle-over-a-torus}
Let $S$ be a complex torus. Let $X$ be a compact complex manifold, and let $\tau\colon X\to S$~ be a flat surjective morphism
whose typical fiber is isomorphic to~$\PP^1$.
Suppose that $X$ is not bimeromorphic to a projectivization of a holomorphic vector bundle of rank $2$ on $S$.
Then the group $\Bim(X)$ is Jordan.
\end{lemma}

Throughout the paper we will frequently use the notion of the \emph{algebraic
reduction} of a compact complex manifold. For such a manifold~$X$,
the algebraic reduction is a meromorphic map $X\dasharrow Y$ with connected fibers to a projective variety
$Y$ of dimension $\dim Y=\ab(X)$ such that the fields of meromorphic functions of $X$ and $Y$ are isomorphic to
each other; we refer the reader to \cite[\S\,I.3]{Ueno1975} for details. Note that in the case when $\dim X=2$,
the algebraic reduction is a holomorphic  map, and its typical fiber is an elliptic curve provided that~\mbox{$\ab(X)=1$}, see~\mbox{\cite[Proposition~VI.5.1]{BHPV-2004}}.

We will need some auxiliary assertions about compact complex surfaces.
The following fact is well-known.

\begin{lemma}[{see for instance \cite[Lemma~2.1]{ProkhorovShramov-SurfacesBFS}}]
\label{lemma:two-intersecting-curves}
Let $S$ be a compact complex surface. Suppose that $S$ contains two
divisors $C_1$ and $C_2$ such that $C_1^2\ge 0$ and~\mbox{$C_1\cdot C_2>0$}.
Then the surface~$S$ is algebraic.
\end{lemma}

\begin{corollary}\label{corollary:two-intersecting-curves}
Let $S$ be a compact complex surface of algebraic dimension~$1$.
Then all the curves on $S$ are contained in the fibers
of its algebraic reduction.
\end{corollary}
\begin{proof}
Suppose that there exists an irreducible curve $Z$ on $S$ that is not
contained in a fiber of the algebraic reduction $\theta$ of $S$.
Let $F$ be a typical
fiber of $\theta$. Then $F^2=0$ and $F\cdot Z>0$. Thus, for
$n\gg 0$ one has $(Z+nF)^2>0$ and
$(Z+nF)\cdot F>0$. Now the assertion follows from
Lemma~\ref{lemma:two-intersecting-curves}.
\end{proof}

Recall that a compact complex surface
$S$ is called \emph{minimal} if it does not contain any smooth rational curves with self-intersection~$-1$.
There exists a Kodaira--Enriques classification of minimal compact complex surfaces,
see~\cite[Chapter~VI]{BHPV-2004}.
Recall in particular that a compact complex surface of non-negative Kodaira dimension
is non-ruled (that is, it is not covered by rational curves).
Any non-algebraic K\"ahler compact complex surface has non-negative Kodaira dimension.

\begin{lemma}[{see for instance \cite[Proposition~3.5]{ProkhorovShramov-Surfaces}}]
\label{lemma:Bir-vs-Aut}
Let $S$ be a non-ruled minimal compact complex surface. Then~\mbox{$\Bim(S)=\Aut(S)$}.
\end{lemma}

\begin{lemma}[{see \cite[Proposition~1.2]{ProkhorovShramov-SurfacesBFS}, \cite[Lemma~2.4]{ProkhorovShramov-SurfacesBFS}}]
\label{lemma:Kodaira-base-subgroup}
Let $S$ be either a Kodaira surface, or a minimal compact complex surface with~\mbox{$\varkappa(S)=1$}.
In the former case define the elliptic fibration~\mbox{$\phi\colon S\to C$} as the algebraic reduction of~$S$.
In the latter case define~$\phi$ as the pluricanonical fibration. Then the image of the group~\mbox{$\Aut(S)$} in~\mbox{$\Aut(C)$} is finite.
\end{lemma}

\begin{proposition}[{see \cite[Theorem~1.1]{ProkhorovShramov-SurfacesBFS}}]
\label{proposition:BFS-for-surfaces}
Let~$S$ be a compact complex surface of non-zero Kodaira dimension.
Suppose that the group~\mbox{$\Bim(S)$} of bimeromorphic selfmaps of $S$
has unbounded finite subgroups. Then~$S$ is bimeromorphic to a surface of one of the following types:
\begin{itemize}
\item a complex torus;

\item a bielliptic surface;

\item a Kodaira surface;

\item a surface of Kodaira dimension $1$.
\end{itemize}
Moreover, in the first three cases the group $\Bim(S)$ always has
unbounded finite subgroups.
\end{proposition}

The following assertion is well-known, but we provide its proof for
the reader's convenience.

\begin{lemma}\label{lemma:pa-a-0}
Let $S$ be a non-algebraic compact complex surface, and let $D$ be a non-zero effective divisor on $S$. Then
$\pa(D)\le 1$.
\end{lemma}

\begin{proof}
If $\ab(S)=1$, then by
Corollary~\ref{corollary:two-intersecting-curves}
every effective divisor~$D$ on~$S$ is contained in the fibers of
the algebraic reduction of~$S$. The latter is an elliptic fibration,
and thus
in this case it is obvious that~\mbox{$\pa(D)\le 1$}. Therefore, we will assume that~\mbox{$\ab(S)=0$}.

Consider the exact sequence of sheaves
$$
0\to \OOO_S(-D)\to \OOO_S\to\OOO_D\to 0.
$$
It gives the exact sequence of cohomology groups
$$
\ldots\to H^1(S, \OOO_S)\to H^1(D, \OOO_D)\to H^2(S, \OOO_S(-D))\to\ldots
$$
Thus we obtain
$$
\h^1(D, \OOO_D)\le \h^1(S,\OOO_S)+\h^2(S,\OOO_S(-D)).
$$
Serre duality implies
$$
\h^1(D, \OOO_D)\le \h^1(S,\OOO_S)+\h^0(S,\KKK_S\otimes\OOO_S(D)).
$$
Since $\ab(S)=0$, the space of global sections of any line bundle on $S$ has dimension at most $1$. Hence
$$
\h^1(D, \OOO_D)\le \h^1(S,\OOO_S)+1.
$$
Recall that
$$
\h^1(S,\OOO_S)=\h^{0,1}(S)\le \h^{1,0}(S)+1,
$$
see for instance \cite[Theorem~IV.2.7]{BHPV-2004}.
Furthermore, since $\ab(S)=0$, one has $\h^{1,0}(S)\le 2$, see
\cite[Proposition~IV.8.1(ii)]{BHPV-2004}.
Therefore, we obtain
\begin{equation}\label{eq:h01}
\h^1(D, \OOO_D)\le \h^{0,1}(S)+1\le \h^{1,0}(S)+2\le 4.
\end{equation}
This gives
\begin{equation}\label{eq:pa}
\pa(D)=\h^1(D,\OOO_D)-\h^0(D,\OOO_D)+1\le \h^1(D,\OOO_D)\le 4.
\end{equation}

Suppose that $\pa(D)>1$. Then~\eqref{eq:pa} implies
$$
\h^1(D, \OOO_D)>1,
$$
and thus $\h^{0,1}(S)>0$ by~\eqref{eq:h01}. In particular, we have~\mbox{$\rk H^1(S,\ZZ)> 0$}. Hence for any positive integer $n$ there exists an unramified
$n$-fold covering
$\pi\colon S'\to S$ (see for instance \cite[Proposition~I.18.1(i)]{BHPV-2004}).
Set~\mbox{$D'=\pi^*(D)$}. Then $D'$ is a non-zero effective divisor on a compact complex surface~$S'$.
Since $S'$ has zero algebraic dimension, the above arguments show that
$\pa(D')\le 4$.
On the other hand, we have
\begin{multline*}
6\ge
2\pa(D')-2=\deg \left(\KKK_{S'}\otimes\OOO_{S'}(D')\right)\vert_{D'}=\\
=\deg \left(\pi^*\KKK_{S}\otimes\pi^*\OOO_{S}(D)\right)\vert_{D'}=\\
=n\cdot\deg \left(\KKK_{S}\otimes\OOO_{S}(D)\right)\vert_{D}=n\cdot(2\pa(D)-2)\ge 2n.
\end{multline*}
The obtained contradiction shows that the inequality $\pa(D)>1$ is impossible.
\end{proof}

\begin{lemma}\label{lemma:pa-a-0-Kahler}
Let $S$ be a compact K\"ahler surface of algebraic dimension~$0$, and let $D$ be a connected
reduced effective divisor on~$S$. Then~\mbox{$\pa(D)=0$}.
\end{lemma}

\begin{proof}
In this case the minimal model $S_{\min}$ of the surface~$S$ is either a complex torus or a $K3$ surface.
If $\pa(D)>0$, then $D$ varies in a positive-dimensional algebraic family. On the other hand, since~\mbox{$\ab(S_{\min})=0$},
the surface $S_{\min}$ contains at most
a finite number of curves, see~\mbox{\cite[Theorem~IV.8.2]{BHPV-2004}}. Hence the support of any connected divisor on~$S_{\min}$
is a tree of smooth rational curves. Therefore, the same holds
for the surface~$S$.
\end{proof}

\begin{remark}\label{remark:not-much-interesting}
Let $S$ be a non-algebraic compact complex surface.
Then $S$ contains at most a finite number of rational curves.
Indeed, if $\ab(S)=0$, then $S$ contains a finite number of
curves at all by~\mbox{\cite[Theorem~IV.8.2]{BHPV-2004}}.
If $\ab(S)=1$, then it follows from
Corollary~\ref{corollary:two-intersecting-curves}
that all the curves on $S$ are contained in the fibers of
its algebraic reduction, and it remains to notice that a typical
fiber of the latter is an elliptic curve.
A similar argument shows that a non-algebraic compact complex surface
contains at most a finite number of singular curves.
\end{remark}

\section{Uniruled manifolds}
\label{section:uniruled}

In this section we recall the definitions and the main properties of uniruled and rationally connected manifolds as well as main properties of rationally connected fibrations.

A compact complex manifold is said to be \emph{uniruled} if it can be covered by rational curves. More precisely, a compact complex manifold~$X$ is  uniruled if
there exist compact complex manifolds $\mathcal{U}$ and $\mathcal{Z}$ and
morphisms
\[
\xymatrix@R=7pt{
&\mathcal{U}\ar[dl]_{\mathrm{\pi}}\ar[dr]^{\mathrm{\varphi}}&
\\
\mathcal{Z}&& X
}
\]
such that a typical fiber of $\pi$ is a smooth rational curve, and the morphism $\varphi$
is surjective and does not contract a typical fiber of the morphism $\pi$.
This is equivalent to the existence of
a compact complex manifold $Y$, a holomorphic rank $2$ vector bundle $\EEE$ on $Y$, and a dominant meromorphic map
$$
f\colon \PP_Y(\EEE) \dashrightarrow X,
$$
which does not factor through the projection
$$
\PP_Y(\EEE)\to Y,
$$
see~\mbox{\cite[Lemma~2.2]{Fujiki1981}}.
It is clear that uniruledness is a bimeromorphic invariant.

\begin{proposition}[{\cite[Remark,~p.~691]{Fujiki1981}, \cite[Corollary~IV.1.11]{Kollar-1996-RC}}]
Let~$X$ be a uniruled compact complex manifold.
Then $H^0(X, (\KKK_X)^{\otimes p})=0$ for all $p>0$, i.\,e.
$\varkappa(X)=-\infty$.
\end{proposition}

\begin{theorem}[\cite{Demailly-Peternell:2003}, \cite{Hoering-Peternell-2016}]
Let $X$ be a compact K\"ahler manifold of dimension at most $3$. Then $X$ is uniruled if and
only if $\varkappa(X)=-\infty$.
\end{theorem}

The following assertion is an analytic version of  well known Abhyankar's result (see \cite[VI.1.2]{Kollar-1996-RC}).

\begin{lemma}
\label{lemma:exc}
Let $f\colon X \dashrightarrow Y$ be a bimeromorphic map of compact complex manifolds, and let $D\subset X$ be an irreducible divisor contracted by~$f$. Then $D$
is bimeromorphic to a uniruled compact complex manifold.
\end{lemma}
\begin{proof}
Replacing $X$ by a resolution of singularities of the graph of the map $f$ (see \cite{Bierstone-Milman-1997}), we may assume that $f$ is holomorphic.
According to the relative complex analytic version of the Chow lemma \cite[Corollary~2]{Hironaka1975}, the map $f$ is dominated by a projective morphism,
and we can replace $f$ by this projective morphism. In other words, we assume that $f$ is a blow up of some coherent ideal sheaf $\FFF$ on~$Y$.
Then it follows from \cite[Theorem~1.10]{Bierstone-Milman-1997} that the morphism $f$ is dominated by a composition of blow ups with smooth centers $f'\colon X'\to Y$.
All exceptional divisors of the latter morphism are uniruled compact complex manifolds.
\end{proof}

A compact complex manifold $X$ is said to be \textit{rationally connected} if
two typical points of $X$ can be connected by a rational curve. More precisely, $X$
is rationally connected if
there exist compact complex manifolds $\mathcal{U}$ and $\mathcal{Z}$ and morphisms
\[
\xymatrix@R=7pt{
&\mathcal{U}\ar[dl]_{\pi}\ar[dr]^{\varphi}&
\\
\mathcal{Z}&& X
}
\]
such that a typical fiber of $\pi$ is a  rational curve, and the induced map
\[
\varphi^2\colon \mathcal{U}\times_\mathcal{Z} \mathcal{U} \longrightarrow X \times X
\]
is surjective. This definition is taken
from \cite[Definition~IV.3.2]{Kollar-1996-RC} and is a little bit
different from \cite[Definition~3.1]{Campana1991a}.
It is easy to see that they are equivalent.
It is also clear that rational connectedness is a bimeromorphic invariant.

\begin{proposition}[{\cite[Corollary~IV.3.8]{Kollar-1996-RC}}]
\label{proposition:RC-no-pluriforms}
Let $X$ be a rationally connected compact complex manifold. Then $X$
carries no global holomorphic (pluri)forms:
\begin{equation*}
H^0(X, (\Omega_X^1)^{\otimes p})=0
\end{equation*}
for all~\mbox{$p>0$}.
\end{proposition}
\begin{proof}
This fact was proved in \cite[Corollary~IV.3.8]{Kollar-1996-RC} for projective algebraic
manifolds. The proof works in the general case as well.
\end{proof}

Let $X$ and $S$ be compact complex manifolds.
A dominant meromorphic map $f\colon X \dashrightarrow S$ is called \emph{rationally connected fibration} if its typical fiber
is irreducible and rationally connected. A rationally connected fibration
is called \emph{maximal} if for its sufficiently general fiber~$X_s$
and sufficiently general point $x\in X_s$ there exists
no rational curve~\mbox{$C\subset X$} that passes through $x$ and is not contained in $X_s$;
here by a sufficiently general point we mean a point from
the complement to the union of a countable number of proper
closed analytic subsets, and by a sufficiently general fiber
we mean a fiber over a sufficiently general point.
If such a map exists, then it is unique;
in particular, it is eqiuivariant with respect to
the action of the group of bimeromorphic automorphisms of the complex manifold.
The maximal
rationally connected fibration  exists for arbitrary compact K\"ahler
manifold~\mbox{\cite[Theorem~2.3, Remark~2.8]{Campana1992}}.

\begin{remark}\label{remark:MRC-basic}
A compact complex
manifold is rationally connected if and only if the base of its maximal rationally connected fibration
is a point.
\end{remark}

Using Proposition~\ref{prop:Kahler-proj},
it is easy to deduce the following result.
Recall that in the category of compact complex spaces there exists
a resolution of singularities (see, e.\,g., \cite{Bierstone-Milman-1997}).

\begin{corollary}
\label{cor:Kahler-proj}
Let $X$ be a compact K\"ahler manifold,
and let~\mbox{$\tau\colon X \dashrightarrow S$}
be a rationally connected fibration. Assume that $X$
is not algebraic.
Then~\mbox{$\h^{2,0}(S)\neq 0$}.
\end{corollary}

\begin{proof}
Replacing $X$ with another bimeromorphic model we may assume that the map $\tau$ is holomorphic.
If \mbox{$\h^{2,0}(S)=0$}, then $\h^{2,0}(X)=0$ (because the fibers of $\tau$ are covered by rational
curves). Therefore,
$X$ is algebraic by Proposition~\ref{prop:Kahler-proj},
which contradicts our assumptions.
\end{proof}

\begin{corollary}\label{cor:Kahler-proj1}
Let $X$ be a non-algebraic compact K\"ahler manifold, and let $\tau\colon X \dashrightarrow S$ be a rationally connected fibration.
Then~\mbox{$\dim S\ge 2$}. In particular, $X$ is not rationally connected.
If $\dim S=2$, then $S$ is a K\"ahler surface with $\pg(S)>0$.
In particular, the Kodaira dimension of $S$ is non-negative.
\end{corollary}
\begin{proof}
According to Corollary~\ref{cor:Kahler-proj}
we have~\mbox{$\h^{2,0}(S)\neq 0$}, and therefore
the dimension of $S$ cannot be less than $2$. In particular, $S$ is not a point,
i.\,e. by Remark~\ref{remark:MRC-basic} the manifold $X$ is not rationally connected.
If $S$ is a surface, then
according to the above the inequality~\mbox{$\pg(S)>0$} holds
and so $\varkappa(S)\ge 0$.
Moreover, the surface $S$ is K\"ahler
by~\mbox{\cite[Theorem~5]{Varouchas1989}},
because  K\"ahlerness
is preserved under bimeromorphic maps of compact complex surfaces.
\end{proof}

The following assertion is a partial generalization of Corollary~\ref{cor:Kahler-proj1}.
Note that for the proofs of the main results of this paper such generality
is not needed.

\begin{proposition}[{\cite{Graber-Harris-Starr-2003}, \cite[\S3]{CampanaWinkelmann2016}}]
\label{proposition:Kahler-proj2}
Let $X$ be a compact K\"ahler manifold,
and let $\tau\colon X \dashrightarrow Y$ be the maximal rationally connected fibration.
Then~$Y$ is not uniruled.
\end{proposition}
\begin{proof}
Let $\phi\colon Y \dashrightarrow Z$ be  the maximal rationally connected fibration of $Y$.
Suppose that $\dim Z<\dim Y$. We may assume that~$\tau$ and~$\phi$ are holomorphic maps.
Also, we may assume that the manifold~$Y$ is K\"ahler, see
\cite[Theorem~5]{Varouchas1989}.
Let $F$ be a typical fiber of the composition~\mbox{$\phi\circ\tau$}.
Then $F$ is a K\"ahler manifold.
Hence there exists a rationally connected
fibration $F\to \tau(F)$ over the rationally connected base~\mbox{$\tau(F)$}.
As in the proof of Corollary~\ref{cor:Kahler-proj}, we obtain
$$
\h^{2,0}(F)=\h^{2,0}(\tau(F))=0.
$$
Thus $F$ is a projective algebraic variety.
According to~\mbox{\cite[Corollary~1.3]{Graber-Harris-Starr-2003}}, the variety $F$ is rationally connected.
Hence, we have the equality $\dim Z=\dim Y$, a contradiction.
\end{proof}

\begin{theorem}[{cf. Corollary~\ref{cor:Kahler-proj1}}]
\label{theorem:RC-Kahler-algebraic}
Let $X$ be a compact K\"ahler rationally connected manifold. Then $X$ is a projective algebraic variety.
\end{theorem}
\begin{proof}
Note that the vector bundle $\Omega_X^2$ is a direct summand of the vector bundle $(\Omega_X^1)^{\otimes 2}$.
Hence it follows from Proposition~\ref{proposition:RC-no-pluriforms}
that~\mbox{$H^{0}(X, \Omega_X^2)= 0$}.
It remains to apply Proposition~\ref{prop:Kahler-proj}.
\end{proof}

There is a classification of
non-algebraic uniruled compact K\"ahler threefolds which describes them in terms of
their maximal rationally connected fibrations.
We do not use this classification, however we provide it here for
completeness.

\begin{theorem}[{\cite{Fujiki1983}}, {\cite[Theorem~9.1]{Campana-Peternell:2000}}, {\cite[Theorem~1.2]{Campana-Horing-Peternell}}]
\label{theorem:CHP}
Let~$X$ be a compact K\"ahler threefold.
Assume that $X$ is not algebraic, and~\mbox{$\varkappa(X)=-\infty$}.
Let $\eta\colon X\dashrightarrow B$ be the algebraic reduction,
and let~\mbox{$\tau\colon X \dashrightarrow S$}~be the maximal rationally connected fibration.
Then~\mbox{$\dim S=2$}, and one of the following cases occurs.
\begin{enumerate}
\item
$\ab(X) = 0$, $\ab (S) = 0$; in this case $S$ is either a complex torus or a $K3$ surface.
\item
$\ab(X) = 1$, $\ab (S) = 0$; in this case $X$ is bimeromorphic to $\PP^1 \times S$, where $S$ is either a complex torus or a $K3$ surface.
\item
$\ab(X) = 1$, $\ab(S)=1$; in this case the
algebraic reduction~\mbox{$\eta\colon X \dasharrow B$}
can be decomposed as follows
\[
\eta\colon X \overset{\tau}\dashrightarrow S \overset{\beta}\dashrightarrow B,
\]
where $\tau$ coincides with the relative Albanese map and
$\beta$ is the algebraic reduction of the surface $S$.
A typical fiber of the map~$\eta$ has the form $\PP (\OOO\oplus \LLL)$, where
$\LLL$ is a non-torsion
line bundle of degree~$0$ over an elliptic curve.

\item
$\ab (X) = 2$,
$\ab(S)=1$.  There is
the following commutative diagram
\[
\xymatrix@R=8pt{
&X\ar@{-->}[dl]_\tau\ar@{-->}[dr]^{\eta}&
\\
S\ar[dr]&&B\ar[dl]
\\
&C&
}
\]
where $C$ is a curve, $B\to C$ is a fibration with typical fiber~$\PP^1$, and $S\to C$ is the algebraic reduction.
The induced map
$$
X \dasharrow S\times_C B
$$
is dominant.
\end{enumerate}
\end{theorem}

\section{Maximal rationally connected fibration}
\label{section:MRC}

In this section
we study the relation between maximal rationally connected fibrations with structure of groups of bimeromorphic
automorphisms of complex manifolds.

For a meromorphic map $\gamma\colon X \dashrightarrow Y$ of
compact complex manifolds,
by $\ind(\gamma)$ we denote its indeterminacy locus.
This is the minimal proper closed analytic subset in $X$ such
that the restriction~\mbox{$\gamma|_{X\setminus \ind(\gamma)}$} is holomorphic.
Recall that the codimension of $\ind(\gamma)$ is at least two.

Given a surjective morphism of compact complex
manifolds~\mbox{$h\colon X\to Z$}, we will consider the subgroup $\Bim(X)_{h}$
in~\mbox{$\Bim(X)$}
that consists of bimeromorphic selfmaps whose action is fiberwise with respect to $h$.
Also, we will consider the subgroup
$\Bim(X)_{h}^{\hol}$ in~\mbox{$\Bim(X)_{h}$}
consisting of bimeromorphic selfmaps $\gamma$
that are holomorphic on a typical fiber of~$h$ (the set of such fibers may depend on~$\gamma$).

\begin{lemma}
\label{lemma:fiberwise-embedding}
Let $X$ and $Z$ be compact complex manifolds, and~\mbox{$h\colon X\to Z$}~be a surjective morphism with connected fibers.
Let~\mbox{$G_i, i\in\mathbb{N}$}, be a countable family of finite subgroups in~\mbox{$\Bim(X)_{h}$}
(respectively, in~\mbox{$\Bim(X)_{h}^{\hol}$}).
Then there exists a smooth, irreducible, and reduced
fiber $F$ of the map $h$ of dimension $\dim X-\dim Z$ such that all the groups $G_i$ are embedded into the group $\Bim(F)$
(respectively, into the group~\mbox{$\Aut(F)$}).
Moreover, if $\dim Z>0$, and we are
given a countable union $\Xi$ of proper closed analytic subsets in $Z$, then the fiber $F$ can be chosen so that the point $h(F)$ does not lie on $\Xi$.
\end{lemma}

\begin{proof}
Let the set $\Delta\subset Z$ consist of those points over which the morphism $h$ is not smooth.

Chose a non-trivial bimeromorphic map $\gamma$
contained in the group~\mbox{$\Bim(X)_{h}$}.
Consider the set $\nabla_\gamma\subset Z$ consisting of all the points~$P$ for which
$\ind(\gamma) \supset h^{-1}(P)$. Thus the map $\gamma$ is defined in a typical point of the fiber
$h^{-1}(P)$ over $P\in Z\setminus \nabla_\gamma$.

Consider also the subset $\Delta_\gamma\subset Z\setminus \nabla_\gamma$ consisting of all the points~$P$
such that for a typical point~\mbox{$Q\in h^{-1}(P)$} one has $\gamma(Q)=Q$.
Then for any point
$$
P\not\in\Delta\cup\overline{\Delta_\gamma}\cup\nabla_\gamma
$$
the map $\gamma$ defines an element $\gamma\vert_F$ of
the group $\Bim(F)$, where~\mbox{$F=h^{-1}(P)$}, and if $\gamma$ is non-trivial, then
the element $\gamma\vert_F$ is also non-trivial. Moreover if $\gamma\in\Bim(X)_h^{\hol}$, then
$$
\gamma\vert_F\in\Aut(F)\subset\Bim(F).
$$

It is obvious that the sets $\Delta$, $\overline{\Delta_\gamma}$, and $\nabla_\gamma$ are proper closed analytic subsets in $Z$. If~\mbox{$\dim Z>0$}, fix also a subset $\Xi$ which is a countable union of proper closed analytic subsets in $Z$.
Since the field $\mathbb{C}$ is uncountable,  $Z$ cannot
be represented as a union of a countable number of proper closed analytic subsets. Hence
the complement
$$
U=Z\setminus\left(\Delta\cup\bigcup_{\gamma\in\bigcup_i G_i, \gamma\neq\mathrm{id}}\left(\overline{\Delta_{\gamma}}\cup\nabla_\gamma\right)\cup\Xi\right)
$$
is non-empty. It remains to note that fiber of the map $h$ over any point from $U$ satisfies the desired properties.
\end{proof}

\begin{corollary}\label{corollary:Jordan-fiberwise}
Let $X$ and $Z$ be compact complex manifolds, and~\mbox{$h\colon X\to Z$}~be a surjective morphism with connected fibers.
Assume that the group $\Bim(X)_{h}$
(respectively, the group~\mbox{$\Bim(X)_{h}^{\hol}$}) is not Jordan.
Then for a typical fiber $F$ of the map $h$, the group
$\Bim(F)$ (respectively, the group~\mbox{$\Aut(F)$})
is not Jordan.
\end{corollary}

\begin{proof}
In the group $\Bim(X)_{h}$
(respectively, in the group~\mbox{$\Bim(X)_{h}^{\hol}$})
one can find a countable number of subgroups
$G_i$, $i\in\mathbb{N}$, such
that all~$G_i$ cannot simultaneously appear as subgroups in any Jordan group.
On the other hand, by Lemma~\ref{lemma:fiberwise-embedding} all the groups
$G_i$ are embedded into the group~\mbox{$\Bim(F)$}
(respectively, into the group~\mbox{$\Aut(F)$})
for a typical fiber $F$ of the map $h$.
\end{proof}

\begin{corollary}\label{corollary:fiberwise-RC}
Let $X$ and $Z$ be compact complex manifolds, and~\mbox{$h\colon X\to Z$}~be a surjective morphism with connected fibers.
Assume that the manifold $X$ is K\"ahler, and
a typical fiber of the map $h$ is rationally connected.
Then the group~\mbox{$\Bim(X)_{h}$} is Jordan.
\end{corollary}
\begin{proof}
Assume that the group
$\Bim(X)_{h}$ is not Jordan.
From Corollary~\ref{corollary:Jordan-fiberwise} we obtain that, for a typical
fiber $F$ of the map $h$, the group~\mbox{$\Bim(F)$} is not Jordan.
On the other hand, the complex manifold $F$ is K\"ahler and
rationally connected, and so it is a projective algebraic
variety by Theorem~\ref{theorem:RC-Kahler-algebraic}.
Thus,
the group $\Bim(F)$ is Jordan by \cite[Theorem~1.8]{ProkhorovShramov-RC}
and \cite[Theorem~1.1]{Birkar}, a contradiction.
\end{proof}

\begin{corollary}\label{corollary:fiberwise-dim-1}
Let $X$ and $Z$ be compact complex manifolds, and~\mbox{$h\colon X\to Z$}~be a surjective morphism with connected fibers.
Assume that a typical fiber of the map $h$ has dimension~$1$.
Then the group $\Bim(X)_{h}$ is Jordan.
\end{corollary}
\begin{proof}
Assume that the group
$\Bim(X)_{h}$ is not Jordan. From Corollary~\ref{corollary:Jordan-fiberwise} we obtain that, for some
smooth irreducible one-dimensional fiber $F$ of the map $h$, the group $\Bim(F)=\Aut(F)$ is not Jordan.
This gives an obvious contradiction.
\end{proof}

\begin{corollary}\label{corollary:fiberwise-dim-2}
Let $X$ and $Z$ be compact complex manifolds, and~\mbox{$h\colon X\to Z$}~be a surjective morphism with connected fibers.
Assume that a typical fiber of the map $h$ has dimension~$2$.
Then the group $\Bim(X)_{h}^{\hol}$ is Jordan.
\end{corollary}
\begin{proof}
Assume that the group
$\Bim(X)_{h}^{\hol}$ is not Jordan. From Corollary~\ref{corollary:Jordan-fiberwise} we obtain that, for some
smooth irreducible two-dimensional fiber $F$ of the map $h$, the group $\Aut(F)$ is not Jordan.
However this is impossible by \cite[Theorem~1.6]{ProkhorovShramov-Surfaces}.
\end{proof}

\begin{corollary}\label{corollary:fiberwise-Kahler}
Let $X$ and $Z$ be compact complex manifolds, and~\mbox{$h\colon X\to Z$}~be a surjective morphism with connected fibers. Assume that the manifold $X$ is K\"ahler.
Then the group $\Bim(X)_{h}^{\hol}$ is Jordan.
\end{corollary}
\begin{proof}
Assume that the group
$\Bim(X)_{h}^{\hol}$ is not Jordan.
From Corollary~\ref{corollary:Jordan-fiberwise}
we obtain that, for some
smooth irreducible fiber $F$ of the map $h$, the group $\Aut(F)$ is not Jordan.
However this is impossible by Theorem~\ref{theorem:aut} because $F$ is a compact K\"ahler manifold.
\end{proof}

Now we prove the main result of this section.

\begin{proposition}\label{proposition:MRC-not-Jordan}
Let $X$ be a compact complex
threefold, and $S$ be a compact complex surface.
Let $\tau\colon X \dashrightarrow S$ be a rationally connected fibration. Assume that the surface $S$ has non-negative Kodaira dimension,
and $S$ is not an algebraic surface.
Finally, assume that the group $\Bim(X)$ is not Jordan.
Then the surface~$S$ is
bimeromorphic to a complex torus.
\end{proposition}

\begin{proof}
Since $\varkappa(S)\ge 0$, the surface $S$ is not ruled.
Therefore, the map~$\tau$ is equivariant with respect to the whole group $\Bim(X)$.
Hence there is an exact sequence of groups
\begin{equation*}
1\to \Bim(X)_{\tau}\to\Bim(X)\to\Bim(S).
\end{equation*}

We may assume that the map $\tau$ is holomorphic.
Since the group~\mbox{$\Bim(X)$} is not Jordan, the group $\Bim(S)$
has unbounded finite subgroups.
This follows from Corollary~\ref{corollary:fiberwise-RC} (or Corollary~\ref{corollary:fiberwise-dim-1}) and Lemma~\ref{lemma:group-theory}.

Since the Kodaira dimension of the surface $S$ is non-negative
and the group $\Bim(S)$ has unbounded finite subgroups, $S$ is bimeromorphic either to a complex torus,
or to a bielliptic surface, or to a Kodaira surface, or to a surface of Kodaira dimension $1$, see
Proposition~\ref{proposition:BFS-for-surfaces}.
Since the surface~$S$ is not algebraic, it cannot be bimeromorphic to a bielliptic surface.
We may assume that the surface $S$ is minimal.

Suppose that either $S$ is a Kodaira surface, or $\varkappa(S)=1$.
In the former case  consider its algebraic reduction
$\phi\colon S\to C$. In the latter case
consider the pluricanonical fibration $\phi\colon S\to C$
(which also coincides with the algebraic reduction under our assumptions).
In both cases put $\psi=\phi\circ \tau$.
These maps form a $\Bim(X)$-equivariant commutative diagram
\begin{equation}\label{eq:psi-phi-h}
\vcenter{
\xymatrix{
X\ar[d]_\tau\ar[rd]^{\psi}& \\
S\ar[r]_{\phi}&C
}}
\end{equation}

Consider the
subgroups $\Bim(X)_{\psi}$ and $\Bim(X)_{\psi}^{\hol}$ in $\Bim(X)$.
Let $\gamma$ be an arbitrary element of $\Bim(X)_{\psi}$.
Since $\ab(S)=1$, by
Corollary~\ref{corollary:two-intersecting-curves}
it does not contain any curves which are surjectively projected to $C$.
Therefore, $\tau(\ind(\gamma))$ is contained in a union of a
finite number of fibers of the map $\phi$,
and $\ind(\gamma)$ is contained in a union of a finite number of fibers of
the map $\psi$. In other words, one has
$\gamma\in\Bim(X)_{\psi}^{\hol}$. This implies that the group $\Bim(X)_{\psi}$
coincides with its subgroup $\Bim(X)_{\psi}^{\hol}$.
In particular, the group $\Bim(X)_{\psi}$ is Jordan by  Corollary~\ref{corollary:fiberwise-dim-2}.

Since the surface $S$ is minimal, one has $\Bim(S)=\Aut(S)$ by Lemma~\ref{lemma:Bir-vs-Aut}.
Consider the subgroup $\Aut(S)_{\phi}$ in $\Aut(S)$ consisting
of those automorphisms of $S$ whose action is fiberwise with respect to $\phi$.
There are exact sequences of groups
\begin{equation*}
1\to \Bim(X)_{\psi}\to\Bim(X)\to \Gamma_{\psi}
\end{equation*}
and
$$
1\to \Aut(S)_{\phi}\to\Aut(S)\to \Gamma_{\phi},
$$
where the groups $\Gamma_{\psi}$ and $\Gamma_{\phi}$ are subgroups in $\Aut(C)$.
Since the diagram~\eqref{eq:psi-phi-h} is commutative, we see that $\Gamma_{\psi}\subset\Gamma_{\phi}$.
On the other hand, the group $\Gamma_{\phi}$ is finite by Lemma~\ref{lemma:Kodaira-base-subgroup}.
Now we conclude from Lemma~\ref{lemma:group-theory} that the group $\Bim(X)$ is Jordan.
The obtained contradiction shows that $S$ is not a Kodaira surface, and $\varkappa(S)\neq 1$, so
$S$ is a complex torus.
\end{proof}

\begin{remark}
If in the notation of Proposition~\ref{proposition:MRC-not-Jordan} the manifold $X$ is K\"ahler,
and $\tau$ is the maximal rationally connected fibration,
then some other hypotheses are satisfied automatically.
Namely, the surface~$S$ has non-negative Kodaira dimension
by Corollary~\ref{cor:Kahler-proj1} (and also $S$ is non-algebraic
by Lemma~\ref{lemma:projective-easy} under the additional
assumption that the dimension of $X$ is equal to~$3$).
Moreover, in this case the surface $S$ is K\"ahler by Corollary~\ref{cor:Kahler-proj1},
and so in the proof we do not need to consider the case where $S$ is a Kodaira surface.
Also, if $X$ is K\"ahler, we can use Corollary~\ref{corollary:fiberwise-Kahler} instead of Corollary~\ref{corollary:fiberwise-dim-2} in the proof of Proposition~\ref{proposition:MRC-not-Jordan}.
\end{remark}

\begin{remark}
If in the notation of Proposition \ref{proposition:MRC-not-Jordan} the surface $S$ is a complex torus,
then~\mbox{$\tau\colon X \dashrightarrow S$}
is nothing but the Albanese map.
In particular, in this case $\tau$ is a holomorphic map.
\end{remark}

\section{Conic bundles over non-algebraic surfaces}
\label{section:conic-bundles}

In this section
we study the properties of conic bundles over non-algebraic compact complex surfaces.
Many results of this section are stated and proved
in greater generality than we need in the present paper.

\begin{definition}\label{definition:conic-bundle}
A proper surjective morphism~\mbox{$f\colon X\to S$} of
complex manifolds is called \textit{a conic bundle} if any fiber of~$f$
is isomorphic to a conic in~$\PP^2$. A conic bundle $f\colon X\to S$ is said to be \textit{standard} if, for any prime divisor $D \subset S$, its inverse image $f^{-1}(D)$ is irreducible.
\end{definition}

\begin{remark}
According to Definition~\ref{definition:conic-bundle},
for a conic bundle~\mbox{$f\colon X\to S$}
the anticanonical line bundle $\KKK_X^{-1}$ is $f$-ample.
In particular, $f$ is a projective morphism.
\end{remark}

For a conic bundle $f\colon X\to S$ the set of points
in which the map~$f$ is not smooth
forms a divisor $\Delta$ on~$S$ which is called the \textit{discriminant divisor} or the \textit{degeneracy divisor}.
The following assertion is well known,
see, e.\,g.,
\cite[Corollary~3.3.3]{P:rat-cb:r}, \cite[(3.3.2)]{P:rat-cb:r}, \cite[Corollary~3.9.1]{P:rat-cb:r}, and~\mbox{\cite[(3.8.2)]{P:rat-cb:r}}.

\begin{lemma}
\label{lemma:discriminant}
Let $f\colon X\to S$ be a standard conic bundle over a compact complex
surface $S$, and let $\Delta$~be its discriminant curve.
Suppose that $\Delta\neq\varnothing$. Then $\Delta$ has only ordinary double singularities.
Moreover,
the fiber of the map $f$ over a point $o\in \Delta$ is reduced (and has two irreducible components)
if and only if $o$ is a smooth point of the curve~$\Delta$.

Assume furthermore that there exists an irreducible
component~\mbox{$\Delta_1\subset \Delta$} which is a smooth rational curve.
Then the intersection of~$\Delta_1$ and~\mbox{$\Delta-\Delta_1$}
is non-empty and consists of an even number of points.
\end{lemma}

\begin{remark-definition}\label{remark:minimal:surface}
Let $S$ be a non-algebraic compact complex surface,
and let $S_{\min}$ be its minimal model.
Then there is
at most a finite number of  singular curves on $S_{\min}$
by Remark~\ref{remark:not-much-interesting}.
Hence there exists a sequence of blowups $\bar S\to S_{\min}$ such that
any  curve on~$\bar S$ has only
ordinary double singularities,
and $\bar S$ satisfies the universal property:
if~$S'$ is a compact complex surface bimeromorphic to $\bar S$
and such that any  curve on $S'$ has only
ordinary double singularities, then one
has a modification~\mbox{$S'\to \bar S$}. In particular, such a surface $\bar S$ is unique.
We call it an \emph{almost minimal surface}, or the
\emph{almost minimal model} of the surface $S$.
\end{remark-definition}

\begin{remark}\label{remark:torus-easy}
For applications in the framework of this paper we need only conic bundles
over surfaces bimeromorphic to complex tori. It is easy to see that
the almost minimal model of any such  surface is exactly the
complex torus.
\end{remark}

\begin{remark}\label{remark2:minimal:surface}
If a compact complex surface $S$ is non-algebraic, then by Lemma~\ref{lemma:pa-a-0}
any curve on $S_{\min}$ (and also on $\bar S$ and $S$) has arithmetic genus at most~$1$.
If furthermore $S$ is K\"ahler and $\ab(S)=0$, then any connected curve on
$S_{\min}$, $\bar S$, and $S$ has arithmetic genus~$0$
by Lemma~\ref{lemma:pa-a-0-Kahler}.
\end{remark}

The projective case of the following assertion was proved in~\cite{Sarkisov-1982-e}.
An alternative proof was given in~\cite{Avilov-conic-r}; it is based on a result concerning 
three-dimensional extremal contraction, see e.g.~\mbox{\cite[Corollary~7.7.1]{MoriProkhorov2019}}.

\begin{proposition}[see {\cite[Proposition 3.8]{Lin2017}}]
\label{proposition:standartization}
Let $X$ be
a compact complex threefold,
and $S$ be a non-algebraic compact complex surface.
Let~\mbox{$f\colon X\dashrightarrow S$} be a
rationally connected fibration.
Let $\bar S$ be the almost minimal model of the surface $S$. Then there exists a commutative diagram
\[
\xymatrix@R=17pt{
X\ar@{-->}[d]_{f}\ar@{-->}[r] & \bar X\ar[d]^{\bar f}
\\
S\ar@{-->}[r]^{} & \bar S
}
\]
where the
horizontal arrows are bimeromorphic maps, and $\bar f\colon \bar X\to \bar S$ is a standard conic bundle. Here the discriminant curve of the conic bundle $\bar f$ is either empty, or is a disjoint union of smooth elliptic curves, rational curves with one ordinary double point, and combinatorial cycles
of smooth rational curves.
\end{proposition}
\begin{proof}
According to \cite[Proposition~3.8]{Lin2017}, there exists a standard conic bundle
$$
f' \colon X' \to S'
$$
fiberwise bimeromorphic to $f$, where $S'$ is some bimeromorphic model of the surface $S$.
Let $\Delta'$ be the discriminant curve of $f'$. By Lemma~\ref{lemma:discriminant} and Remark~\ref{remark2:minimal:surface}
each connected component $\Delta^{\prime(i)}\subset\Delta'$ is either
a smooth elliptic curve, or a rational curve with one ordinary double point,
or a combinatorial cycle of smooth rational curves.

According to Remark-Definition~\ref{remark:minimal:surface}, we have a modification
$$
\sigma\colon S' \to \bar S.
$$
Assume that $S'\neq \bar S$.
Then the
$\sigma$-exceptional locus contains at least one $(-1)$-curve $C'$. Again according to
Lemma~\ref{lemma:discriminant} and Remark~\ref{remark2:minimal:surface},
the curve $C'$ either does not intersect the discriminant divisor $\Delta'$, or is contained in $\Delta'$, and then
$$
C'\cdot (\Delta'-C')=2.
$$
In this case we can apply to $X^\sharp$ a sequence of bimeromorphic transformations described in \cite[Lemma~4]{Iskovskikh-mat-cb-1991-ru} or \cite[Proposition~8.5]{P:rat-cb:r}, contract all $(-1)$-curves on $S'$ and obtain a standard conic bundle over~$\bar{S}$.
\end{proof}

\begin{remark}
Proposition~\ref{proposition:standartization} cannot be directly generalized to the case of algebraic
surfaces. For example, over any projective algebraic surface~$S$ there exist non-standard conic bundles
which are not bimeromorphic to standard conic bundles over a minimal model
of~$S$.
\end{remark}

Proposition~\ref{proposition:standartization} implies the following

\begin{corollary}\label{corollary:aX-aS-1}
Let $X$ be a compact complex threefold,
and $S$ be a non-algebraic compact complex surface
bimeromorphic to a complex torus~$S_0$.
Let
$\tau\colon X\dasharrow S$ be a
rationally connected fibration.
Assume that
$X$ is not bimeromorphic to the projectivization of a holomorphic rank $2$ vector bundle on $S_0$.
Then the group $\Bim(X)$ is Jordan.
\end{corollary}

\begin{proof}
By Proposition~\ref{proposition:standartization} and Remark~\ref{remark:torus-easy} we may assume
that the surface $S=S_0$ is a complex torus,
and the map $\tau$ is holomorphic and it is a conic bundle.
In this case the desired assertion follows from Lemma~\ref{lemma:P1-bundle-over-a-torus}.
\end{proof}

The following proposition refines some results obtained in \cite{Lin2017}.
We do not use it in the proofs of the main results of the present paper.

\begin{proposition}\label{proposition:Delta-to-Delta}
Let $f\colon X\to S$ be a standard conic bundle,
where $X$ is a compact complex
threefold, and $S$ is an almost minimal non-algebraic compact complex surface. Let $\Delta\subset S$ be the discriminant curve of the conic bundle $f$. Then any bimeromorphic map~\mbox{$\varphi\colon X\dashrightarrow X$} fits into the following commutative diagram
\begin{equation}
\label{eq:diagram:map}
\vcenter{
\xymatrix@R=17pt{
X\ar[d]_f\ar@{-->}[r]^{\varphi} &X\ar[d]^f
\\
S\ar[r]^{\delta} &S
}}
\end{equation}
where $\delta$ is an isomorphism. Moreover, the map~$\varphi$ cannot contract components of the divisor $f^{-1}(\Delta)$
and so~\mbox{$\delta(\Delta)=\Delta$}.
\end{proposition}

\begin{proof}
Since
the surface $S$ is non-algebraic, the number of rational curves on $S$
is finite by Remark~\ref{remark:not-much-interesting}. Therefore,
$f\colon X\to S$ is the maximal rationally connected
fibration. Hence the map $\varphi$ is fiberwise, i.\,e. we have the diagram \eqref{eq:diagram:map} with a bimeromorphic map $\delta$. According to the universal property (see Remark-Definition~\ref{remark:minimal:surface}) this map is an isomorphism.

For convenience, we write our map as follows:
$$
\varphi\colon X' \dashrightarrow X'',
$$
where $X'=X''=X$.
Put
$$
f'=\delta \circ f,\quad \Delta'=\delta(\Delta),\quad f''=f,\quad \Delta''=\Delta.
$$
Thus, $\Delta'$ (respectively, $\Delta''$) is the discriminant curve of the conic bundle
$f'\colon X'\to S$ (respectively,~\mbox{$f''\colon X''\to S$}), and the
diagram~\eqref{eq:diagram:map} is rewritten as the following
commutative diagram
\[
 \xymatrix@R=13pt {
 X'\ar[dr]_{f'}\ar@{-->}[rr]^{\varphi}&&X''\ar[dl]^{f''}
 \\&S&
 }
\]
Here the action of $\varphi$ is fiberwise.
Assume that $\varphi$ contracts an irreducible surface~\mbox{$D'=f^{\prime -1}(\Lambda)$}, where $\Lambda$~is an irreducible component of the discriminant curve $\Delta'$.
Then the map $\varphi^{-1}$ must contract an irreducible surface~\mbox{$D''=f^{\prime\prime -1}(\Lambda)$}.
It follows from the commutativity of the diagram  that
$$
\Theta'=\varphi^{-1}(D'')\subset X'
$$
is an irreducible curve, and  $f'(\Theta')=\Lambda$.

Let $o$ be a typical point of the curve $\Lambda$,
and let~\mbox{$C'=f^{\prime -1}(o)$} and~\mbox{$C''=f^{\prime\prime -1}(o)$}
be the fibers over the point $o$ of the maps~$f'$ and~$f''$, respectively.
Since $o$ is a smooth point of the discriminant curve~$\Delta'$,
by Lemma~\ref{lemma:discriminant} the conic~$C'$ has two irreducible
components~$C_1'$ and~$C_2'$ meeting transversally at one point.
If $\Theta'$ meets only one of the component of the fiber  $C'$, then the same holds for nearby fibers over points of~$\Lambda$.
Consider the double cover parameterizing irreducible components of the fibers over the
points of $\Lambda$, that is, the cover by the corresponding component of the Douady space,
which is compact by \cite{Fujiki1982}, because the (singular) surface $D'=f^{\prime -1}(\Lambda)$ is bimeromorphic to
a ruled surface (and in particular is Moishezon).
According to the above, this double cover splits. We conclude that
the surface $D'$ is reducible, which gives a contradiction.

Therefore,
both curves~$C_1'$ and $C_2'$  must intersect the curve~$\Theta'$, hence they must intersect the indeterminacy locus~\mbox{$\ind(\varphi)\supset \Theta'$}.
On the other hand, we have an inclusion~\mbox{$\ind(\varphi)\supset \Theta'$}.
Indeed, otherwise the map $\varphi$ is defined at a typical point of the curve~$\Theta'$.
Considering the graph
\[
 \xymatrix@R=11pt {
 &\hat X\ar[dl]_{p'}\ar[dr]^{p''}&
 \\
 X'\ar@{-->}[rr]^{\varphi}&&X''
 }
\]
of the map $\varphi$, we see that the map $p'^{-1}$ is defined at a typical point of the curve $\Theta'$, and its inverse morphism $p'$
contracts the irreducible divisor $\hat D\subset\hat X$, that is the proper transform of the divisor $D''$ with respect to~$p''$, on the curve $\Theta'$.
Clearly, this is impossible.
Therefore, one has~\mbox{$\ind(\varphi)\supset \Theta'$}, and thus both curves $C_1'$ and $C_2'$  intersect $\ind(\varphi)$.

Consider the germ of an analytic
curve $\Upsilon \subset S$ transversally meeting~$\Lambda$ at $o$.
Consider the surfaces
$$
V'=f^{\prime -1}(\Upsilon)\subset X',\qquad V''=f^{\prime\prime -1}(\Upsilon)\subset X''.
$$
Thus, $V'$ is the inverse image of $V''$ with respect to
the map $\varphi$, and we have a bimeromorphic map $\varphi_V\colon V' \dashrightarrow V''$. The natural projections
$$
f_V'\colon V'\to\Upsilon, \quad
f_{V}''\colon V''\to \Upsilon
$$
are fibrations whose typical fiber is a smooth rational curve, and the
curves~\mbox{$C'=V'\cap D'$} and $C''=V''\cap D''$ are fibers over
the point $o$ of the maps~$f_V'$ and $f_V''$, respectively. Note that the curve $C''$ is contracted by the map $\varphi_V^{-1}$, so that the map $\varphi_V$
is not holomorphic. On the other hand, its indeterminacy locus $\ind(\varphi_V)$ is contained in $C'$; moreover,
in a neighborhood of  $C'$ the set $\ind(\varphi_V)$ coincides with $\ind(\varphi)$.
Consider a resolution of indeterminacies:
\begin{equation}\label{eq:Vprime}
\vcenter{
\xymatrix@R=13pt {
&\tilde V\ar[dl]_p\ar[dr]^q&\\
V'\ar[dr]_{f_V'}\ar@{-->}[rr]^{\varphi_V}&& V''\ar[dl]^{f_V''}\\
&\Upsilon&
}}
\end{equation}
We may assume that this resolution is minimal, i.\,e. it has the minimal possible  Picard number.
Then none of the $(-1)$-curves on $\tilde V$ can be simultaneously
contracted by both maps~$p$ and~$q$.
Hence there exists a $(-1)$-curve $\tilde E$ which is $q$-exceptional but not $p$-exceptional.
Since the map $\varphi_V^{-1}$ is an isomorphism on the complement to the
curve~$C''$,
we conclude from the commutativity of the diagram~\eqref{eq:Vprime}
that~\mbox{$p(\tilde E)$} is contained in $C'$.
Then $p(\tilde E)$ coincides
with one of the irreducible components of $C'$, say, one
has~\mbox{$p(\tilde E)=C_1'$}.
Since
$$
C_1^{\prime 2}=-1=\tilde E^2,
$$
the map $p$ must be an isomorphism near $C_1'$.
Thus, the indeterminacy locus
$\ind(\varphi_V)= \varphi_{V}^{-1}(C'')$
does not intersect $C_1'$. Therefore, the indeterminacy
locus $\ind(\varphi)$, which coincides
with~$\varphi^{-1}(D'')$ in a neighborhood of~$C'$,
also does not intersect $C_1'$.
But this contradicts the observation made above.
Hence the map $\varphi$ cannot contract the divisor $f^{\prime -1}(\Lambda)$. This proves the proposition.
\end{proof}

\section{Case $\ab(S)=0$}
\label{section:a-0}

In this section we consider compact complex threefolds
having a structure of a rational curve fibration  over a surface of algebraic dimension~$0$.

\begin{remark}
\label{remark:X-vs-X-prime}
Let
$f\colon X\to S$ be a standard conic bundle over a compact K\"ahler surface
of algebraic dimension $0$. Then $f$ is a smooth morphism such that
all the fibers of~$f$ are isomorphic to~$\PP^1$,
see~\mbox{\cite[Proposition~3.10]{Lin2017}}.
For proofs of our results we do not need  this assertion: we will use a weaker (but more general)
Proposition~\ref{proposition:standartization}.
\end{remark}

The following assertion is a particular case of
\cite[Corollary~3.1]{BZ19}. We provide its proof for the
reader's convenience.

\begin{lemma}
\label{lemma:X-vs-X-prime}
Let $S$ be a compact complex surface that does not contain any curves
(and in particular has
algebraic dimension~$0$). Let~\mbox{$f\colon X\to S$}
be a standard conic bundle.
Then the group $\Bim(X)$ acts biholomorphically on
$X$.
\end{lemma}

\begin{proof}
Since the surface $S$ does not contain any (rational) curves, the map~$f$ is the maximal rationally connected fibration.
In particular, $f$ is equivariant with respect to the group~\mbox{$\Bim(X)$}.
Consider an arbitrary element~\mbox{$\varphi\in \Bim(X)$}.
It induces a biholomorphic map~\mbox{$\delta\colon S\to S$} by Lemma~\ref{lemma:Bir-vs-Aut} (because the surface $S$ is minimal).
As in the proof of Proposition~\ref{proposition:Delta-to-Delta} write our map as follows: $\varphi\colon X' \dashrightarrow X''$, where~\mbox{$X'=X''=X$}.
Put
$$
f'=\delta \circ f,\qquad f''=f.
$$
We have the following commutative diagram
\[
 \xymatrix@R=13pt {
 X'\ar[dr]_{f'}\ar@{-->}[rr]^{\varphi}&&X''\ar[dl]^{f''}
 \\&S&
 }
\]
where the map $\varphi$ is fiberwise.

Assume that   the map $\varphi$ contracts an irreducible divisor~\mbox{$D'\subset X'$}.
Then~$D'$ is bimeromorphic to a ruled surface (see Lemma~\ref{lemma:exc}). This implies that~\mbox{$f'(D')\neq S$}.
Then since $S$ does not contain any curves, the image~\mbox{$f'(D')$} must be a point; this is impossible,
because all the fibers of~$f'$ are one-dimensional.

Thus, $\varphi$ does not contract any divisors. The same holds for the inverse map $\varphi^{-1}$.
Therefore, $\varphi$ is an isomorphism in codimension  $1$, i.\,e.
there exist  closed analytic subsets $Z'\subset X'$ and $Z''\subset X''$ of codimension $2$ such that the restriction
$$
\varphi_{U'}=\varphi\vert_{U'}\colon U'\dashrightarrow U''
$$
to the open subset $U'=X\setminus Z'$ is an isomorphism with the open subset~\mbox{$U''=X\setminus Z''$}.
From the commutativity of the diagram we obtain~\mbox{$f'(Z')=f''(Z'')$}.

Put $\Xi=f'(Z')$.
Since $S$ does not contain any curves,  $\Xi$ is a finite subset on $S$. Put $S_0=S\setminus \Xi$.
The sheaves $\EEE'=f'_*\KKK_X'^{-1}$ and~\mbox{$\EEE''=f''_*\KKK_X''^{-1}$} are locally free of rank $3$.
Moreover, there are natural isomorphisms
\begin{equation*}
\EEE'\vert _{S_0}=  f'_* \KKK_X'^{-1}\vert _{U'}=f'_*\comp \varphi^* \KKK_X''^{-1}\vert _{U'}\cong
 \delta^*\comp f''_* \KKK_X''^{-1}\vert _{U''}=  \delta^* \EEE''\vert _{S_0}.
\end{equation*}

Thus, the vector bundles $\EEE'$ and $\delta^*\EEE''$ coincide on the open subset~\mbox{$S_0\subset S$}
whose complement is zero-dimensional. We claim that they coincide  everywhere, i.\,e.
$\EEE'=\delta^*\EEE''$. Indeed, the problem is local along the base, so we may assume that  $S\ni s$
is a small analytic neighborhood of a point $s$, and our vector bundles are trivial: $\EEE'=S\times \CC^3$ and~\mbox{$\EEE''=S\times \CC^3$}.
The isomorphism~\mbox{$\EEE\cong \EEE''$} on $S_0=S\setminus \{s\}$ is given by a matrix $\| g_{i,j}\|_{1\le i, j\le 3}$ whose entries are holomorphic on   $S_0$ functions. By Hartogs's extension theorem
these can be uniquely extended to  holomorphic functions on   $S$.
Moreover, the matrix $\| g_{i,j}\|$ is invertible on~$S_0$ and so it is invertible also on the whole  $S$.
This shows that there is an isomorphism of vector bundles $\EEE'\stackrel{\sim}\to \EEE''$

Since the anticanonical bundle  $\KKK_X^{-1}$ is very ample relative to $f'$ and~$f''$, it induces embeddings
$i'\colon X \hookrightarrow  \PP_S(\EEE')$ and $i''\colon X \hookrightarrow  \PP_S(\EEE'')$.
Note that there are isomorphisms
$$
\varphi_{U'}^*\KKK_X^{-1}\vert_{U''}\cong \varphi_{U'}^*\KKK_{U''}^{-1}\cong \KKK_{U'}^{-1}\cong\KKK_X^{-1}\vert_{U'}.
$$
This implies that the restrictions
$$
i'\vert_{U'}\colon U' \hookrightarrow  \PP_{S_0}(\EEE'), \quad
i''\vert_{U''}\colon U'' \hookrightarrow  \PP_{S_0}(\EEE'')
$$
of embeddings $i'$ and $i''$
commute with isomorphisms $\varphi_{U'}\colon U'\stackrel{\sim}\to U''$ and~\mbox{$\PP_{S_0}(\EEE') \cong \PP_{S_0}(\EEE'')$}. Thus, thus there is the following commutative diagram
\begin{equation*}
\xymatrix@R=9pt{
& U' \, \ar@{^{(}->}[rr]\ar@{^{(}->}[dl] \ar[dd]^(.35){\cong}|\hole&& X \ar@{^{(}->}[ld]\ar@{-->}[dd]^{\varphi}
\\
\PP_{S_0}(\EEE')\, \ar@{^{(}->}[rr]\ar[dd]^{\cong}  && \PP_{S}(\EEE')\ar[dd]^(.35){\cong}|\hole
\\
& U''\,  \ar@{^{(}->}[rr]\ar@{^{(}->}[dl]&& X\ar@{^{(}->}[dl]
\\
\PP_{S_0}(\EEE'')\, \ar@{^{(}->}[rr]&& \PP_{S}(\EEE'')
}
\end{equation*}
which shows that $\varphi$ is a holomorphic map.
Applying similar arguments to $\varphi^{-1}$ we obtain that
 $\varphi$ is an isomorphism.
\end{proof}

\begin{remark}
In Lemma~\ref{lemma:X-vs-X-prime} it is not sufficient to assume that
the surface $S$ has algebraic dimension~$0$. Indeed, if $S$ contains a
curve $C$, then passing to a suitable bimeromorphic model of $S$ we may suppose that $C$ is non-singular. In this case the conic bundle $X=S\times\PP^1$ admits
elementary transformations over the curve~$C$.
\end{remark}

Now we can to prove the main result of this section.
Note that it can be obtained as a particular case of
\cite[Theorem~1.2]{BZ19}.

\begin{theorem}\label{theorem:aS=0}
Let $X$ be a compact complex
threefold, and $S$ be a compact complex surface with $\varkappa(S)\ge 0$ and $\ab(S)=0$. Let
$\tau\colon X \dashrightarrow S$ be a rationally connected fibration.
Then the group $\Bim(X)$ is Jordan.
\end{theorem}

\begin{proof}
By Proposition~\ref{proposition:MRC-not-Jordan}
it is sufficient to consider the case where the surface $S$ is bimeromorphic to a complex torus.
Moreover,
by Proposition~\ref{proposition:standartization} and Remark~\ref{remark:torus-easy}
we may assume
that $S$ is a complex torus, and the map $\tau$ is holomorphic. Furthermore,
by Lemma~\ref{lemma:P1-bundle-over-a-torus}
it is sufficient to consider the case when $X$ is the projectivization of a holomorphic rank $2$ vector bundle on $S$. In this case the manifold $X$ is K\"ahler by Theorem~\ref{theorem:Voisin}.
Recall that a complex torus  of algebraic dimension zero does not contain any curves. Thus, according to Lemma~\ref{lemma:X-vs-X-prime},
the group~\mbox{$\Bim(X)$} acts biholomorphically on $X$.
Now the assertion follows from Theorem~\ref{theorem:aut}.
\end{proof}

\section{Case $\ab(S)=1$}
\label{section:a-1-1}

In this section we prove Theorem~\ref{theorem:main}.
Our nearest purpose is to study compact complex threefolds of algebraic dimension~$2$ for which the base of the maximal rationally connected fibration
has algebraic dimension~$1$.
The proof of the following fact was explained to us by F.\,Campana.

\begin{lemma}
\label{lemma:fiber-product}
Let $X$ be a compact complex threefold,
$B$ and $S$ be compact complex surfaces, and $C$ be a smooth curve.
Let~\mbox{$\tau\colon X\dasharrow S$}
and~\mbox{$\eta\colon X\dasharrow B$} be dominant meromorphic maps,
wherein $\eta$ has connected fibers.
Let~\mbox{$\theta\colon S\to C$} and $\sigma\colon B\to C$
be surjective morphisms with connected fibers such that $\theta\circ\tau=\sigma\circ\eta$.
Assume that the surface $B$ algebraic, and the
algebraic dimension of the surface~$S$ is equal to~$1$.
Then the manifold~$X$
is bimeromorphic to the fiber product~\mbox{$Y=S\times_C B$}.
\end{lemma}
\begin{proof}
We may assume that the maps~$\tau$
and $\eta$ are holomorphic.
There is a natural morphism $\zeta\colon X\to Y$.
It is easy to see that $Y$ is a three-dimensional
reduced irreducible compact complex space (with at worst
hypersurface singularities), and the morphism $\zeta$
is surjective.
In particular, a typical fiber of $\zeta$ is finite.
Denote by $\tau_Y$ and $\eta_Y$ the natural projections of $Y$ to surfaces
$S$ and $B$, respectively.
\begin{equation}\label{eq:diagram-fiber-product}
\vcenter{
\xymatrix{
&&X\ar[d]^{\zeta}\ar[lldd]_{\tau}\ar[rrdd]^{\eta}&&\\
&&Y\ar@{->}[lld]^{\tau_Y}\ar@{->}[rrd]_{\eta_Y}&&\\
S\ar@{->}[rrd]_{\theta}&&&&B\ar@{->}[lld]^{\sigma}\\
&&C&&
}}	
\end{equation}

Assume that the map $\zeta$ is not bimeromorphic.
Let $R\subset Y$ be the ramification divisor of
$\zeta$.
If $\eta_Y$ does not map $R$ to $B$ surjectively (in particular, if~\mbox{$R=\varnothing$}), then
the map~\mbox{$\eta=\eta_Y\circ\zeta$}
has non-connected fibers over the points of the open
set~\mbox{$B\setminus \eta_Y(R)$}, which
is impossible by our assumption. Let
$R'$~be some irreducible component of the divisor~$R$
which is mapped surjectively to $B$. Then the restriction of the
map~$\eta_Y$ to $R'$ is finite over a typical point of $B$,
which implies that $R'$ is an algebraic surface. Since
the surface $S$ is not algebraic, the restriction of the map
$\tau_Y$ to $R'$ cannot be finite over a typical point
of $S$. Therefore,
the image~\mbox{$\tau_Y(R')$} is contained in a curve on~$S$.

Note that $\ab(S)=1$, and $\theta$ is the
algebraic reduction for $S$. Hence
none of the curves on $S$
maps surjectively to $C$ by
Corollary~\ref{corollary:two-intersecting-curves}.
From this
we see that the image of $\theta\circ\tau_Y(R')$
is a point. On the other hand, the morphism~\mbox{$\sigma\circ\eta_Y$}  surjectively maps $R'$ to $C$ which gives a contradiction with the commutativity of the diagram~\eqref{eq:diagram-fiber-product}.
\end{proof}

\begin{corollary}
\label{corollary:a2MRC1}
Let $X$ be a compact complex threefold of algebraic dimension~$2$,
and $S$ be a compact complex surface bimeromorphic to a complex torus.
Assume that $\ab(S)=1$.
Let ~\mbox{$\tau\colon X\dasharrow S$} be a
rationally connected fibration.
Then the manifold $X$ is bimeromorphic to~\mbox{$S\times\PP^1$}.
\end{corollary}

\begin{proof}
Consider the algebraic reduction~\mbox{$\eta\colon X\dasharrow B$},
where $\dim B=2$ by our assumption. We may assume that the maps $\tau$
and $\eta$ are holomorphic, and $B$ is a non-singular surface.
Moreover, by Proposition~\ref{proposition:standartization} and Remark~\ref{remark:torus-easy} we may assume
that the surface $S$
is a complex torus. Then its algebraic reduction $\theta\colon S\to C$ is an elliptic fibration over some elliptic curve~$C$.

Since $\eta$ does not contract a typical fiber of the map $\tau$,
the surface~$B$ is ruled. The embedding of the fields of meromorphic functions
$$
\MMM(C) \subset \MMM(X)=\MMM(B)
$$
induces a map $\sigma\colon B\dashrightarrow C$ which must be a morphism with rational fibers
(in fact, the
morphisms~\mbox{$\tau\colon X\to S$} and $\sigma\colon B\to C$ are Albanese maps).
By Lemma~\ref{lemma:fiber-product}
the manifold $X$ is bimeromorphic to the fiber product
$S\times_C B$.
Since the ruled surface $B$ is bimeromorphic to ~\mbox{$C\times\PP^1$}, the threefold $X$ is bimeromorphic to $S\times\PP^1$.
\end{proof}

\begin{corollary}
\label{corollary:summary}
Let $X$ be a compact complex threefold,
and $S$ be a non-algebraic compact complex surface of
non-negative Kodaira dimension.
Let ~\mbox{$\tau\colon X\dasharrow S$}~be a
rationally connected fibration.
Assume that the group $\Bim(X)$ is not Jordan.
Then the surface $S$ is bimeromorphic to a complex torus, and the threefold
$X$ is bimeromorphic to the projectivization of a holomorphic rank $2$ vector bundle on this complex torus.
Moreover, if $\ab(X)=2$, then the threefold $X$ is bimeromorphic to the
product~\mbox{$S\times\PP^1$}.
\end{corollary}

\begin{proof}
By Proposition~\ref{proposition:MRC-not-Jordan}
the surface $S$ is bimeromorphic to some complex torus $S_0$, and $\ab(S)=1$ by Theorem~\ref{theorem:aS=0}.
From Corollary~\ref{corollary:aX-aS-1}
we conclude that $X$ is bimeromorphic to the projectivization of a holomorphic rank $2$ vector bundle on a complex torus~$S_0$.
Moreover, if~\mbox{$\ab(X)=2$}, then by Corollary~\ref{corollary:a2MRC1}
the threefold $X$ is bimeromorphic to the
product $S\times\PP^1$.
\end{proof}

Now we are ready to prove the main result of this paper.

\begin{proof}[The proof Theorem~\ref{theorem:main}]
Let $\tau\colon X\dasharrow S$ is the maximal rationally connected fibration.
Since the manifold $X$ is not algebraic, we have~\mbox{$\dim S=2$}
and $\varkappa(S)\ge 0$ by Corollary~\ref{cor:Kahler-proj1}. Also we know that the surface $S$ is not algebraic by Lemma~\ref{lemma:projective-easy}.
It remains to apply Corollary~\ref{corollary:summary}.
\end{proof}

\bibliography{bib}
\bibliographystyle{ugost2008s}
\end{document}